\documentclass[12pt]{article}
\usepackage{amsmath, amscd, amssymb, latexsym, epsfig, color, amsthm, authblk}
\numberwithin{equation}{section}

\newtheorem{theorem}{Theorem}[section]
\newtheorem{proposition}[theorem]{Proposition}

\newtheorem{corollary}[theorem]{Corollary}
\newtheorem{conjecture}[theorem]{Conjecture}
\newtheorem{lemma}[theorem]{Lemma}

\theoremstyle{definition}

\newtheorem{remark}[theorem]{Remark}

\newtheorem*{lbt}{The Lower Bound Theorem}
\DeclareMathOperator{\skel}{Skel}
\DeclareMathOperator\lk{\mathrm{lk}}
\DeclareMathOperator\st{\mathrm{st}}

\DeclareMathOperator{\conv}{\mathrm{conv}}

\DeclareMathOperator{\rank}{rank}

\newcommand{\R}{{\mathbb R}}

\newcommand{\Stress}{{\mathcal S}}

\newcommand{\C}{{\mathcal C}}

\newcommand{\p}{{\mathbf p}}
\newcommand{\m}{{\mathbf m}}
\newcommand{\q}{{\mathbf q}}

\newcommand{\halfd}{\lfloor \frac{d}{2} \rfloor}
\DeclareMathOperator{\Rig}{Rig}

\title{A lower bound theorem for centrally symmetric simplicial polytopes}

\author[1]{Steven Klee}
\author[2]{Eran Nevo}
\author[3]{Isabella Novik}
\author[4]{Hailun Zheng\thanks{Research of Klee is partially supported by NSF grant DMS-1600048, of Nevo by Israel Science Foundation grant ISF-1695/15 and by grant 2528/16 of the ISF-NRF Singapore joint research program, of Novik by  NSF grant DMS-1361423 and DMS-1664865, and of Zheng by graduate fellowships from NSF grant DMS-1361423}}

\affil[1]{\small{Department of Mathematics, Seattle University, 901 12th Avenue, Seattle, WA 98122, USA}}
\affil[2]{\small{Einstein Institute of Mathematics, The Hebrew University of Jerusalem, Jerusalem 91904, Israel}}
\affil[3]{\small{Department of Mathematics, University of Washington, Box 354350, Seattle, WA 98195, USA}}
\affil[4]{\small{Department of Mathematics, University of Michigan, 530 Church Street, Ann Arbor, MI 48109, USA}}
\begin{document}

\maketitle


\begin{abstract}
Stanley proved that for any centrally symmetric simplicial $d$-polytope $P$ with $d\geq 3$, $g_2(P) \geq {d \choose 2}-d$. We provide a characterization of centrally symmetric simplicial $d$-polytopes with $d\geq 4$ that satisfy this inequality as equality.  This gives a natural generalization of the classical Lower Bound Theorem for simplicial polytopes to the setting of centrally symmetric simplicial polytopes.
\end{abstract}

\section{Introduction}

An important invariant in the study of face numbers of simplicial $d$-polytopes and, more generally $(d-1)$-dimensional simplicial complexes, is $g_2:=f_1-df_0+{d+1 \choose 2}$, where $f_1$ and $f_0$ denote the number of edges and the number of vertices, respectively. In this paper we study this invariant for the class of {\em centrally symmetric} simplicial polytopes. We write {\em cs} for centrally symmetric. Our main result is a characterization of cs simplicial $d$-polytopes for which $g_2$ is minimized. The motivation for this work is the classical Lower Bound Theorem.

\begin{lbt}
Let $P$ be a simplicial $d$-polytope with $d\geq 4$.  Then $g_2(P) \geq 0$, with equality if and only if $P$ is stacked.
\end{lbt}

A polytope is \textit{stacked} if it can be obtained from the $d$-simplex by repeatedly attaching (shallow) $d$-simplices along facets.  The $d=4$ case of the Lower Bound Theorem is due to Walkup \cite{Walkup}. The nonnegativity of $g_2$ for arbitrary $d$ was originally proved by Barnette \cite{Barnette-LBT}.  Billera and Lee \cite{Billera-Lee} proved that the equality $g_2(P) = 0$ holds if and only if $P$ is stacked. In fact, as was established in works of Walkup \cite{Walkup}, Barnette \cite{Barnette-LBT-pseudomanifolds}, Kalai \cite{Kalai-rigidity}, Fogelsanger \cite{Fogelsanger}, and Tay \cite{Tay}, the same result holds in the generality of all $(d-1)$-dimensional simplicial complexes whose geometric realizations are closed, connected manifolds or even normal pseudomanifolds.

Much less is known for cs simplicial complexes. Stanley \cite{Stanley-cs} (answering an unpublished conjecture of Bj\"orner) proved that if $P$ is a cs simplicial $d$-polytope ($d \geq 3$), then $g_2(P) \geq {d \choose 2} - d$, and more generally  that $g_r(P) \geq {d \choose r} - {d \choose r-1}$ for all $1 \leq r \leq \halfd$. However, in the thirty years since then, a characterization of cs simplicial $d$-polytopes with $g_2 = {d \choose 2} - d$ (for $d\geq 4$) has not been established, nor has any progress been made on whether the inequality $g_2(P) \geq {d \choose 2} - d$ continues to hold for cs simplicial spheres.

The goal of this paper is to at least partially remedy this situation by characterizing  cs simplicial $d$-polytopes for which $g_2 = {d \choose 2} - d$.  The characterization strongly parallels that of the classical non-cs case:  In the classical setting, a $d$-simplex has the minimal number of faces among all simplicial $d$-polytopes and the stacking operation does not change $g_2$.  In the cs case, the $d$-dimensional cross-polytope has the minimal number of faces among all cs $d$-polytopes, but (arbitrary) stacking may destroy the condition of central symmetry.  However, the  \textit{symmetric stacking} operation, i.e., repeatedly attaching simplices along \textit{antipodal pairs} of facets, will preserve both central symmetry and $g_2$.  We will show that any cs simplicial $d$-polytope for which $g_2 = {d \choose 2}-d$ is obtained from the cross-polytope in this way.  We will use $\C^*_d$ to denote the $d$-dimensional cross-polytope.

\begin{theorem} \label{main-thm}
Let $P$ be a cs simplicial $d$-polytope with $d \geq 4$.  Then $g_2(P) = {d \choose 2}-d$ if and only if $P$ is obtained from $\C^*_d$ by symmetric stacking.
\end{theorem}

As in the classical case, the part of the theorem asserting that polytopes obtained from $\C^*_d$ through symmetric stacking have $g_2 = {d \choose 2}-d$ is immediate from the fact that, for $d\geq 4$, the stacking operation does not affect $g_2$ and $g_2(\C^*_d) = {d \choose 2}-d$. Thus, in the rest of the paper we concentrate on the other implication. The tools we use are from the rigidity theory of frameworks. The vertices and edges of a convex simplicial $d$-polytope provide a framework in $\R^d$ that is infinitesimally rigid by a theorem of Whiteley \cite{Whiteley-84}. Furthermore, it follows from work of Stanley \cite{Stanley-cs} and Lee \cite{Lee-94}, along with more recent work of Sanyal, Werner, and Ziegler {\cite[Thm.~2.1]{Sanyal-et-al}}, that if $P$ is a cs simplicial $d$-polytope with $g_2(P)= {d \choose 2}-d$, then {\em all} stresses on $P$ must be symmetric (see Section \ref{section:3}). Our main strategy in proving Theorem \ref{main-thm} will be to use the symmetry of stresses to understand the missing faces of $P$ and its links.

The rest of the paper is structured as follows. In Section 2, after reviewing basic definitions related to simplicial complexes and simplicial polytopes, we introduce the rigidity theory of frameworks and summarize several important results on the infinitesimal rigidity of polytopes. In Section 3, we establish the lower bound on $g_2$ for rigid cs frameworks. In Section 4, we state a key technical result, Theorem \ref{main-pre-thm}, and prove Theorem \ref{main-thm} under the assumption that Theorem \ref{main-pre-thm} holds. Then in Sections 5, 6, and 7 we establish a sequence of results that lead to a proof of Theorem \ref{main-pre-thm}. We close in Section 8 with some open questions.

\section{Preliminaries}
\subsection{Polytopes and simplicial complexes}

An (abstract) {\em simplicial complex} $\Delta$ with vertex set $V=V(\Delta)$ is a {non-empty} collection of subsets of $V$ that is closed under inclusion. The elements of $\Delta$ are called {\em faces}. The dimension of a face $\tau\in\Delta$ is $\dim\tau:=|\tau|-1$, and the dimension of $\Delta$, $\dim\Delta$, is the maximum dimension of any of its faces. The \textit{facets} of $\Delta$ are maximal faces of $\Delta$ under inclusion. We say that $\Delta$ is \textit{pure} if all of its facets have the same dimension. One example of a simplicial complex on $V$ is the $(|V|-1)$-dimensional simplex $\overline{V}:=\{\tau \ : \ \tau\subseteq V\}$; another example is  the {\em boundary} of this simplex defined as $\partial\overline{V}:=\overline{V}\setminus\{V\}$.

If $\tau$ is a face of a simplicial complex $\Delta$, then the {\em star of $\tau$} and the  {\em link of $\tau$ in $\Delta$} are defined as $\st_\Delta(\tau)=\st(\tau):=\{\sigma\in\Delta \ : \  \sigma\cup\tau\in\Delta\}$ and $\lk_\Delta(\tau)=\lk(\tau):= \{\sigma\in \st(\tau) \ : \ \sigma\cap\tau=\emptyset\}$, respectively. For a vertex $v$ of $\Delta$, we write $\st(v)$ and $\lk(v)$ instead of  $\st(\{v\})$ and $\lk(\{v\})$. { Following terminology introduced by Perles, see for instance \cite{AltPer}, we say that a} set $\sigma \subseteq V(\Delta)$ is a {\em missing face} of $\Delta$ if $\sigma$ is not a face of $\Delta$, but every proper subset of $\sigma$ is a face;  a {\em missing facet} of $\Delta$ is a missing face of size $1+\dim\Delta$. A pure simplicial complex $\Delta$ is \emph{prime} if it does not have any missing facets. { (Missing faces are also known in the literature as empty simplices, minimal non-faces, and hollow faces.)}

Most of {the} simplicial complexes we will consider arise from polytopes. All polytopes considered in this paper are {\em convex} polytopes.  We refer our readers to Ziegler's book \cite{Ziegler} for more background on this fascinating field. Recall that a {\em face} of a polytope $P$ is the intersection of $P$ with a supporting hyperplane. We denote by $V(\tau)$ the vertex set of a face $\tau$ of $P$.

To any simplicial complex $\Delta$ there is an associated topological space $\Vert\Delta\Vert$ called {a} \emph{geometric realization} of $\Delta$.  A $(d-1)$-dimensional simplicial complex $\Delta$ is a \emph{simplicial $(d-1)$-sphere} (respectively, a \emph{simplicial $(d-1)$-ball}) if its geometric realization $\Vert\Delta\Vert$ is homeomorphic to a sphere (respectively, a ball) of dimension $d-1$.  If $P$ is a {\em simplicial} $d$-polytope (i.e., all proper faces of $P$ are geometric simplices), then  the collection of the vertex sets of all the faces of $P$ (except $P$ itself) is a {simplicial $(d-1)$-sphere} called the {\em boundary complex} of $P$; it is denoted by $\partial P$. When talking about the stars and the links of faces in $P$ we mean the stars and the links of the corresponding faces in $\partial P$. Thus, for a face $\tau$ of $P$, $\st_P(\tau)$ and $\lk_P(\tau)$ are a simplicial ball and simplicial sphere, respectively. If $P$ is fixed or understood, we will simply write $\st(\tau)$ and $\lk(\tau)$.

The link of $\tau$ in $P$ is the boundary complex of a polytope. When $\tau=\{u\}$ is a vertex, one such polytope is obtained by { intersecting $P$ with a hyperplane that strictly separates $u$ from the other vertices}; this polytope is called the {\em vertex figure of~$u$}.

If a simplicial $d$-polytope $P$ is the union of two simplicial $d$-polytopes $Q$ and $R$ that share a common facet $\tau$ but whose interiors are disjoint, we write $P=Q\#_{\tau} R$, or simply $P=Q\# R$; in this case $\partial P$ is the usual connected sum of $\partial Q$ and $\partial R$, glued along the boundary of $\tau$: $\partial Q\#_{\partial\overline{\tau}} \partial R$. A simplicial $d$-polytope $P$ is called {\em stacked} if there are $d$-simplices $S_1, S_2,\ldots,S_t$ such that $P=S_1\#\cdots\# S_t$. The boundary complex of a stacked $d$-polytope is called a \emph{stacked $(d-1)$-sphere}.

If $\Gamma$ and $\Delta$ are simplicial complexes on disjoint vertex sets, their \textit{join} is the simplicial complex $\Gamma*\Delta = \{\sigma \cup \tau \ : \ \sigma \in \Gamma \text{ and } \tau \in \Delta\}$. When $\Gamma = \{\emptyset, \{u\}\}$ consists of a single vertex, we write $u * \Delta$ to denote the \textit{cone} over $\Delta$.

A $d$-polytope $P\subset \R^d$ is \emph{centrally symmetric}, or {\em cs} for short, if $P=-P$; that is, $x \in P$ if and only if $-x \in P$. In the same spirit, a simplicial complex $\Delta$ is \textit{centrally symmetric} or {\em cs} if it is endowed with a {\em free involution} $\alpha: V(\Delta) \rightarrow V(\Delta)$ that induces a free involution on the set of all non-empty faces of $\Delta$ (i.e., $\alpha(\tau)\in\Delta$ and $\alpha(\tau)\neq \tau$ for all nonempty faces $\tau \in \Delta$). For brevity, we write $\alpha(\tau)=-\tau$ and refer to $\tau$ and $-\tau$ as {\em antipodal faces}.
One example of a cs $d$-polytope is the {\em cross-polytope}, $\C^*_d$, defined as the convex hull of $\{\pm p_1,\pm p_2,\ldots, \pm p_d\}$; here $p_1,p_2,\ldots, p_d$ are points in $\R^d$ whose position vectors form a basis for $\R^d$.

\subsection{Infinitesimal rigidity of frameworks}

This section is a summary of some notions and results  pertaining to graph rigidity. Asimow and Roth provide a very readable introduction to this subject in \cite{AsimowRothI} and \cite{AsimowRothII}; see also Lee's notes on the $g$-theorem \cite[Section 6]{Lee-notes}.

Let {$G = (V(G),E(G))$} be a finite graph.  A map $\p: V(G) \rightarrow \mathbb{R}^d$ is called a \textit{$d$-embedding} of $G$, or just an \textit{embedding} of $G$ if $d$ is fixed or understood.   The graph $G$, together with a $d$-embedding $\p$, is called a \textit{framework} in $\R^d$, where the edges are viewed as rigid bars and the vertices are viewed as joints.

An \emph{infinitesimal motion} of $\R^d$ is a {continuous} map $\mathbf{\Psi}: \R^d\to \R^d$ such that for any two points $x,y\in \R^d$,
\[\frac{d}{dt}\Bigr|_{t=0}\big\|(x+t\mathbf{\Psi}(x))-(y+t\mathbf{\Psi}(y))\big\|^2=0.\]
In fact, every infinitesimal motion $\mathbf{\Psi}$ of $\R^d$ has the form $\mathbf{\Psi}(x)=Ax+b$, where $A$ is a $d\times d$ orthogonal matrix and $b$ is a translation vector. Similarly, an infinitesimal motion of a framework $(G,\p)$ is a map $\m: V(G)\to \R^d$ such that for any edge $\{u,v\}$ in $G$,
\[\frac{d}{dt}\Bigr|_{t=0}\big\|(\p(u)+t\m(u))-(\p(v)+t\m(v))\big\|^2=0.\]
A framework $(G, \p)$ is called \emph{infinitesimally $d$-rigid} if every infinitesimal motion $\m$ of $(G,\p)$ is induced by some infinitesimal motion $\mathbf{\Psi}$ of $\R^d$, that is, $\m=\mathbf{\Psi}\circ \p$. The prefix ``d-" is sometimes omitted when the context is clear.

{We let $f_0(G):=|V(G)|$ and $f_1(G):=|E(G)|$.} The \emph{rigidity matrix} $\Rig(G, \p)$ of a framework $(G, \p)$ is defined as follows: it is an $f_1(G) \times df_0(G)$ matrix with rows labeled by edges of $G$ and columns grouped in blocks of size $d$, with each block labeled by a vertex of $G$; the row corresponding to $\{u,v\}\in E(G)$ contains the vector $\p(u)-\p(v)$ in the block of columns corresponding to $u$, the vector $\p(v)-\p(u)$ in columns corresponding to $v$, and zeros everywhere else.

A \textit{stress} on $(G,\p)$ is an assignment of weights $\omega = (\omega_e \ : \ e \in E(G))$ to the edges of $G$ such that for each vertex $v$, $$\sum_{u\ : \ \{u,v\} \in E(G)} \omega_{\{u,v\}}(\p(v)-\p(u)) = \mathbf{0}.$$ It follows from the above definitions that stresses on $(G,\p)$ correspond to elements in the kernel of $\Rig(G, \p)^T$, that is, stresses can be viewed as linear dependences among the rows of the rigidity matrix. We denote the space of all stresses on $(G,\p)$ by $\Stress(G,\p)$.

The following fundamental fact is an easy consequence of the Implicit Function Theorem (see \cite{AsimowRothI} and \cite{AsimowRothII}).

{
\begin{lemma}  \label{basic-rig-prop} Let $(G, \p)$ be a framework in $\R^d$ that does not lie in a hyperplane of $\R^d$. Then the following statements are equivalent:
\begin{enumerate}
\item  $(G, \p)$ is infinitesimally rigid in $\R^d$;
\item   $\rank \Rig(G,\p)=df_0(G)-{d+1 \choose 2}$;
\item $\dim \Stress(G,\p) =f_1(G)-df_0(G)+{d+1 \choose 2}$.
\end{enumerate}
\end{lemma}
}

If $(G,\p)$ is a framework in $\R^d$ and $K$ is a subgraph of $G$, we will adopt the (somewhat imprecise) convention of using $(K,\p)$ to denote the restriction of the framework to $K$.  Since $\p$ is defined on $V(G)$, which is a larger set of vertices than $V(K)$, this will not cause any problems. Two standard results in the rigidity theory --- the Gluing and the Cone Lemmas --- will be handy (see, for instance, \cite[Thm.~2]{AsimowRothII} and \cite[Cor.~6.12]{Lee-notes} for the Gluing Lemma, and \cite[Cor.~1.5]{Tay-et-al} for the Cone Lemma).

\begin{lemma} {\rm{(The Gluing Lemma)}} \label{gluing-lemma}
Let $G$ and $G'$ be graphs, and let $(G \cup G',\p)$ be a framework in $\R^d$.  If $(G, \p)$ and $(G', \p)$ are infinitesimally rigid and have $d$ affinely independent vertices in common (i.e., the framework $(G \cap G', \p)$ affinely spans a subspace of dimension at least $d-1$), then $(G \cup G', \p)$ is infinitesimally rigid.
\end{lemma}

Let { $G=(V(G),E(G))$} be a graph and $u\notin V$ a new vertex. The \textit{cone} over $G$ is the graph { $u*G:=(V(G)\cup\{u\}, E(G)\cup\{\{u,v\} \ : \  v\in V\})$.} The following is a special case of \cite[Cor.~1.5]{Tay-et-al}.

\begin{lemma} {\rm{(The Cone Lemma)}}  \label{cone-lemma}
 Let $(u*G, \p)$ be a framework in $\R^d$, and let $\pi$ be either a central projection from $\p(u)$ onto a hyperplane $H$ not containing $\p(u)$ or an orthogonal projection onto a hyperplane $H$ perpendicular to $\p(u)\neq 0$. Assume further that $\pi$ is injective on $(u*G,\p)$. Then $(u*G, \p)$ is infinitesimally rigid in $\R^d$ if and only if $(G, \pi\circ \p)$ is infinitesimally rigid in $H\cong \R^{d-1}$.\end{lemma}

\subsection{Infinitesimal rigidity of polytopes}

The relevance of  framework rigidity to the study of face numbers of simplicial polytopes (pioneered by Kalai in \cite{Kalai-rigidity}) is  evident from Lemma \ref{basic-rig-prop} and the following fundamental result due to Whiteley \cite{Whiteley-84}. For a simplicial complex $\Delta$, we use the notation $(\Delta,\p)$ to say that $\p: V(\Delta) \rightarrow \R^d$ is a framework on the underlying graph of $\Delta$, {i.e., the $1$-dimensional skeleton of $\Delta$}; further, for a simplicial polytope $P$, we write $(P,\p)$ instead of $(\partial P,\p)$. We also denote by $G(Q)$ the graph of $Q$ whenever $Q$ is a simplicial complex or a polytope.

\begin{theorem}[{\rm Whiteley, 1984}]  \label{Whiteley-thm}
Let $P\subset \R^d$ be a simplicial $d$-polytope $P$, where $d\geq 3$. The graph of $P$ with its natural embedding is infinitesimally rigid in $\R^d$.
\end{theorem}

The case $d=3$ of this theorem is due to Dehn.
Whiteley's proof for $d\geq 4$ is by induction on $d$ with the following lemma serving as the main part of the inductive step. As we frequently rely on this lemma, we sketch its proof for completeness.

\begin{lemma}\label{lemma: star-of-face-rigid}
Let $d \geq 4$, let $P$ be a simplicial $d$-polytope with its natural embedding $\p$ in $\R^d$.  Then for every face $\tau$ of $P$ with $1\leq |V(\tau)|\leq d-3$, the framework $(\st_P(\tau),\p)$ is infinitesimally rigid.
\end{lemma}
\begin{proof}
	Let $V(\tau)=\{v_1,\dots,v_k\}$, let $H$ be a hyperplane so that $Q=P\cap H$ is a vertex figure of $v_1$, and let $\q$ be the natural embedding of $Q$ in $H$. Then $(Q, \q)$ is infinitesimally rigid in $H$ because $Q$ is a simplicial $(d-1)$-polytope and $d-1\geq 3$. Since the framework $(Q, \q)$ is the image of $(\lk_P(v_1),\p)$ under the central projection from $v_1$ onto $H$, and since $\st_{P}(v_1)=v_1*\lk_{P}(v_1)$, the $|V(\tau)|=1$ case of the statement follows from the Cone Lemma.
		
For $|V(\tau)|=k>1$, we induct on $k$. Let $\tau'$ be the face of $Q$ formed by the vertices corresponding to $v_2,\ldots,v_k$. Since $\tau'$ has only $k-1\leq d-4$ vertices and since $Q$ is a $(d-1)$-polytope, our inductive hypothesis implies that $(\st_Q (\tau'), \q)$ is infinitesimally rigid in $H$. The fact that $(\st_P(\tau), \p)=(v_1*\st_Q (\tau'),\p)$ together with the Cone Lemma completes the proof. \end{proof}

Combining Theorem \ref{Whiteley-thm} with Lemma \ref{basic-rig-prop} and the equality part of the Lower Bound Theorem gives the following rigidity-theoretic interpretation of the equality part, which is the overarching theme in Kalai's paper \cite{Kalai-rigidity}.

\begin{proposition}
Let $P$ be a simplicial $d$-polytope with $d \geq 4$. The following conditions are equivalent:
\begin{enumerate}
\item $g_2(P) = 0$;
\item $P$ is stacked;
\item the graph of $P$ with its natural embedding in $\R^d$ does not admit any nontrivial stresses.
\end{enumerate}
\end{proposition}

The following result was established by Kalai \cite{Kalai-rigidity} in the context of generic rigidity theory, but the proofs hold for a specific infinitesimally rigid embedding of a graph as well.

\begin{lemma} \label{lemma: missing-face-stress}
Let $P$ be a simplicial $d$-polytope with its natural embedding $\p$ in $\R^d$, and assume that $d \geq 4$.
\begin{enumerate}
\item If the graph of $P$ contains a chordless cycle $C=(u,v,w,z)$, and $e$ is an edge of $C$, then there is a stress on $(P,\p)$ that is non-zero on $e$.
\item Let $\tau$ be a missing face of $\partial P$ with $3 \leq |\tau| \leq d-1$, $e$ an edge in $\tau$, and  $\tau' := \tau \setminus e$. Then there is a stress on $(\st(\tau')\cup\{e\}, \p)$ that is non-zero on $e$.
\end{enumerate}
\end{lemma}
\begin{proof} (Sketch) For {part} 1, let $e = \{u,v\}$ and $e'=\{w,z\}$.  Then $(\st(w)\cup \st(z), \p)$ is infinitesimally rigid. (Indeed, the stars $(\st(w), \p)$ and $(\st(z), \p)$ are infinitesimally rigid by Lemma \ref{lemma: star-of-face-rigid} and share $d$ affinely independent vertices, namely, the vertices of any facet of $P$ that contains $e'$.) Furthermore,  since $C$ is a chordless cycle in the graph of $P$,  $e$ is a missing edge of  $\st(w)\cup \st(z)$. It then follows from Lemma \ref{basic-rig-prop} that the matrices $\Rig(\st(w)\cup \st(z), \p)$ and $\Rig(\st(w)\cup \st(z) \cup\{e\}, \p)$ have the same rank. Hence the $e$-row of the latter matrix is a linear combination of the other rows. The statement follows.

For part 2, note that $1\leq |\tau'|\leq d-3$, and so $(\st(\tau'), \p)$ is infinitesimally rigid by Lemma~\ref{lemma: star-of-face-rigid}. Since $e$ is a missing edge of this star, the same argument as above completes the proof.
\end{proof}

The statement of {part} 1 in Lemma \ref{lemma: missing-face-stress} can be extended to chordless cycles of length $k \geq 4$.  We only include the proof for $k=4$ here since the proof is shorter and that is the only case we require for this paper.

\section{Rigidity theory for centrally symmetric graphs} \label{section:3}

In this section we will couple rigidity theory with central symmetry to establish lower bounds for rigid frameworks that respect central symmetry.

Recall that if $\Delta$ is a $(d-1)$-dimensional simplicial complex, then $g_2(\Delta)=f_1(\Delta) - df_0(\Delta) + {d+1 \choose 2}$. Similarly, if $(G,\p)$ is any $d$-framework that affinely spans $\R^d$, { we define} $$g_2(G,\p):= f_1(G) - df_0(G) + {d+1 \choose 2}.$$  When $\p$ is clear, we will simply write $g_2(G)$ in place of $g_2(G,\p)$; we will only employ the notation $g_2(G,\p)$ when the dimension of the ambient space in which the graph is embedded is unclear.

We say that $(G,\p)$ is a {\em cs $d$-framework}  if the graph $G$ is cs and the embedding $\p: V(G) \rightarrow \R^d$ respects the symmetry; i.e., $\p(-v) = -\p(v)$ for all $v \in V(G)$.
If $(G,\p)$ is a cs framework, define
$$\Stress^{\text{sym}}(G,\p):= \{\omega\in\Stress(G,\p) \, : \, \omega_e=\omega_{-e} \quad \mbox{for all edges $e$ of $G$}\}.$$

Our key tool will be the following rigidity-theoretic result for cs frameworks.  The result and proof are practically
identical to that of Sanyal et al.~{\cite[Thm.~2.1]{Sanyal-et-al}} (there they work only with cs polytopes, but here we state the result for general rigid cs frameworks), so we only give a summary that highlights the part of the proof that will be relevant for our later results.

\begin{lemma} \label{lemma: g2-lower-bound}
Let $d\geq 3$, and let $(G,\p)$ be an infinitesimally rigid cs $d$-framework that affinely spans $\R^d$. Then $g_2(G) \geq {d \choose 2}-d$.
\end{lemma}

\begin{proof}
The computations of  \cite[p. 188--189]{Sanyal-et-al} apply verbatim to give the following inequality, which is Eq.~(8) in \cite{Sanyal-et-al}:
\begin{equation} \label{eq: symm-stress-dim}
\dim \Stress^{\text{sym}}(G,\p) \geq \frac{f_1(G)}{2} - \frac{d f_0(G)}{2} + {d \choose 2}.
\end{equation}
Since $(G,\p)$ is infinitesimally $d$-rigid, $g_2(G) = \dim \Stress(G,\p)$, and hence
\begin{eqnarray*}
g_2(G) &=& f_1(G) - d f_0(G) + {d+1 \choose 2} \\
&=& {\mbox{\small $2\left[f_1(G)-d f_0(G) + {d+1 \choose 2} \right]- \left[f_1(G)-d f_0(G) + 2{d \choose 2}\right] + 2{d \choose 2}-{d+1 \choose 2}$}} \\
&\overset{(*)}{\geq}& 2\dim \Stress(G,\p) - 2\dim \Stress^{\text{sym}}(G,\p) + 2{d \choose 2}-{d+1 \choose 2} \\
&\overset{(**)}{\geq} & 2{d \choose 2}-{d+1 \choose 2} \\
&=& {d \choose 2}-d.
\end{eqnarray*}
\noindent Here, the inequality (*) comes from Eq.~\eqref{eq: symm-stress-dim} and the inequality (**) follows from the fact that $\Stress^{\text{sym}}(G,\p)$ is a subspace of $\Stress(G,\p)$.
\end{proof}

The computation at the end of the proof of Lemma \ref{lemma: g2-lower-bound} shows that if $g_2(G) = {d \choose 2} -d$  then $\Stress(G,\p) = \Stress^{\text{sym}}(G,\p)$.  This proves the following important corollary.

\begin{corollary}\label{cor: symmetric-stresses}
Let $(G,\p)$ be an infinitesimally rigid cs $d$-framework with $d \geq 3$ that affinely spans $\R^d$.  If $g_2(G) = {d \choose 2}-d$ then every stress on $(G,\p)$ is symmetric.
\end{corollary}

The following result is another immediate consequence of Lemma \ref{lemma: g2-lower-bound}.

\begin{corollary} \label{inf-rig-subgraph}
Let $d \geq 3$ and let $(G,\p)$ be an infinitesimally rigid cs $d$-framework with $g_2 = {d \choose 2}-d$. If $G'$ is a subgraph of $G$ such that $(G',\p)$ is cs, infinitesimally $d$-rigid, and affinely spans $\R^d$, then $g_2(G') = {d \choose 2}-d$ and $\Stress(G',\p) = \Stress(G,\p)$.
\end{corollary}

\begin{proof}
Since $(G', \p)$ is a subframework of $(G,\p)$, $\Stress(G',\p) \subseteq \Stress(G,\p)$. Further, since both frameworks are infinitesimally rigid and cs, Lemma \ref{lemma: g2-lower-bound} implies that
\[
{d \choose 2}-d \leq \dim \Stress(G',\p) \leq \dim \Stress(G,\p)={d \choose 2}-d,
\]
and the statement follows.
\end{proof}

\section{Proof of the main result}

In this section we prove our main result. Following the custom, we write $g_2(P)$ instead of $g_2(\partial P)$.

\begin{theorem} \label{main-thm'}
Let $P$ be a cs simplicial $d$-polytope with $d \geq 4$ satisfying $g_2(P) = {d \choose 2}-d$.  Then $P$ can be obtained from the $d$-dimensional cross-polytope by symmetric stacking operations.
\end{theorem}

The proof of Theorem \ref{main-thm'} relies on a key technical result, Theorem \ref{main-pre-thm} below.  We will state that result in this section and then use it to prove the main result.  In Sections \ref{section:5}, \ref{section:6} and \ref{section:7}, we will establish a sequence of lemmas that will ultimately be used to prove Theorem \ref{main-pre-thm}.

First, we reduce the problem to the case that $P$ is prime and $d \geq 5$. Recall that a simplicial polytope $P$, which is not a simplex, is {\em prime} if $\partial P$ has no missing facets.

\begin{lemma} \label{no-missing-facets}
Let $P$ be a cs simplicial $d$-polytope with $d \geq 4$ satisfying $g_2(P) = {d \choose 2}-d$. If $\partial P$ contains a missing facet, then $P$ can be decomposed as $$P = Q \# P' \# (-Q),$$ where $Q$ is a stacked $d$-polytope and $P'$ is a cs simplicial $d$-polytope satisfying $g_2(P') = {d \choose 2}-d$. In particular, $P$ can be obtained from a prime cs simplicial $d$-polytope $R$ satisfying $g_2(R) = {d \choose 2}-d$ through symmetric stacking.
\end{lemma}

\begin{proof}
Let $\tau = \{v_1,\ldots,v_d\}$ be a missing facet in $\partial P$.  Then $-\tau = \{-v_1,\ldots,-v_d\}$ is also a missing facet in $\partial P$.  Cutting $P$ along the affine span of the vertices of $\tau$ and along the affine span of the vertices of $-\tau$ gives a decomposition of $P$ as $Q \# P' \# Q'$, so that $Q' = -Q$ and $P'$ is cs and simplicial.  Thus $g_2(Q)$ and $g_2(Q')$ are nonnegative and $g_2(P') \geq {d \choose 2}-d$.  But $${d \choose 2}-d = g_2(P) = g_2(Q)+g_2(P') + g_2(Q') \geq 0 + {d \choose 2}-d + 0.$$ Hence $g_2(P') = {d \choose 2}-d$ and $g_2(Q) = g_2(Q') = 0$, which implies $Q$ and $Q'$ are stacked by the Lower Bound Theorem. The in particular part of the statement now follows by induction on the number of missing facets of $\partial P$.
\end{proof}

\begin{proposition}
Let $P$ be a cs prime simplicial $4$-polytope with $g_2(P) = {4 \choose 2}-4 = 2$.  Then $P$ is a cross-polytope.
\end{proposition}

\begin{proof}
Since $P$ is prime and $g_2(P) = 2$, it follows from Theorem 5.4 in \cite{Zheng-g2} that either $P$  is  $\C^*_4$ or $\partial P$ can be obtained from the boundary complex of a simplicial $4$-polytope by performing a stellar subdivision at a $2$-dimensional face. In the former case we are done. In the latter case, let $v$ be the new vertex introduced by this stellar subdivision and note that $\lk_P(v)$ is the suspension of the boundary of a triangle. Since $P$ is cs, vertex $v$ has an antipodal vertex $-v$ whose link is isomorphic to the link of $v$. Suppose $\lk_P(v)$ is the suspension of the cycle on vertices $\{a,b,c\}$ so that $\lk_P(-v)$ is the suspension of the cycle on vertices $\{-a,-b,-c\}$.  Let $\Sigma$ be the simplicial sphere obtained from $\partial P$ by performing a stellar weld at $v$ and $-v$ (i.e., remove $v$ and $-v$ together with their incident edges, fill in the triangles $\{a,b,c\}$ and $\{-a,-b,-c\}$ and join them with their suspending vertices).  This creates a new {\em cs} simplicial $3$-sphere  with $g_2(\Sigma) = 0$. By Walkup's result \cite{Walkup}, $\Sigma$ must be the boundary complex of a stacked $4$-polytope, which is impossible as a stacked polytope cannot be cs.
\end{proof}

Now we only need to establish Theorem \ref{main-thm'} for prime cs simplicial polytopes of dimension $d \geq 5$. Before we can complete the proof, we state our main technical theorem. { For a simplicial complex $\Delta$, we denote the graph of $\Delta$ (i.e., the $1$-dimensional skeleton of $\Delta$) by $G(\Delta)$.}

\begin{theorem} \label{main-pre-thm}
Let $P$ be a prime cs $d$-polytope with $d \geq 5$ and $g_2(P) = {d \choose 2}-d$.  For every vertex $u$ of $P$,
\begin{enumerate}
\item the complexes $\st(u)$ and $\st(-u)$, and hence also $\lk(u)$ and $\lk(-u)$, share exactly $2d-2$ vertices; and
\item the graphs of $\st(u) \cup \st(-u)$ and $P$ coincide: $G(\st(u) \cup \st(-u)) = G(P)$.
\end{enumerate}
\end{theorem}

The proof of this theorem is surprisingly involved.  First, we show $\st(u)$ and $\st(-u)$ share exactly $2d-2$ vertices (see Lemma \ref{lemma: not-d-antipodal-pairs} and Proposition \ref{prop: antipodal-pairs-in-links}).  This will in turn imply that the graph of $\st(u) \cup \st(-u)$, with its embedding in $\R^d$ induced by the vertex coordinates of $P$, is cs and infinitesimally $d$-rigid (see Corollary \ref{cor: star-union-rigid}), and hence that $G(\st(u) \cup \st(-u))$ is exactly the graph of $P$ (see Proposition \ref{prop: cs-rigid-is-whole-graph}).  However, assuming Theorem \ref{main-pre-thm} holds, we are now in a position to complete the proof of our main result.

\begin{proposition}
Let $P$ be a prime cs $d$-polytope with $d \geq 5$ and $g_2(P) = {d \choose 2}-d$.  Then $P$ is a cross-polytope.
\end{proposition}

\begin{proof}
Let $u$ be a vertex of $P$. Let $\deg(u)$ denote the degree of $u$ in the graph of $P$. Then
\begin{eqnarray}
2  \deg(u) &=& \deg(u) + \deg(-u) \nonumber \\
&=& (f_0(P)-2) + f_0(\lk(u) \cap \lk(-u)) \nonumber \\
&=& f_0(P)-2 + 2d-2 = f_0(P) + 2d-4. \label{eq: deg-ineq}
\end{eqnarray}
Here, the second line follows from Theorem \ref{main-pre-thm}(2) which implies that every vertex in $V(P) \setminus \{u,-u\}$ is adjacent to either $u$ or $-u$; furthermore, the vertices in $\lk(u) \cap \lk(-u)$ (i.e., the common neighbors of $u$ and $-u$) are counted twice.  The third line follows from Theorem~\ref{main-pre-thm}(1).

Summing Eq. \eqref{eq: deg-ineq} over all vertices yields
\begin{equation} \label{handshake}
4f_1(P) = \sum_{u \in V(P)} 2 \deg(u) = f_0(P)\cdot (f_0(P) + 2d-4).
\end{equation}
The fact that $g_2(P) = {d \choose 2} - d$ implies that $f_1(P) = d \cdot f_0(P) - 2d$.  Substituting this into Eq.~\eqref{handshake} we conclude that
\begin{eqnarray*}
4d \cdot f_0(P) - 8d &=& \left(f_0(P)\right)^2 + (2d-4)f_0(P), \\
\text{ or equivalently, \qquad } 0 &=& (f_0(P)-4)(f_0(P)-2d).
\end{eqnarray*}
Thus $f_0(P) = 2d$.  The result follows from the fact that the $d$-dimensional cross-polytope is the {\em only} cs $d$-polytope with exactly $2d$ vertices.
\end{proof}

\section{Finding symmetric subgraphs in $G(P)$}  \label{section:5}

 Let $P\subset \R^d$ be a cs simplicial $d$-polytope with $g_2(P) = {d \choose 2}-d$. Without loss of generality (we may perturb the vertices of $P$  without changing the symmetry or combinatorial type of $P$), we assume for the rest of the paper that {\bf every $d$ vertices of $P$, no two of which are antipodal, are affinely independent}, and that $\p : V(\partial P)\to \R^d$ is given by the vertex coordinates of $P$.

In this section we use Lemma \ref{lemma: g2-lower-bound} and Corollary \ref{cor: symmetric-stresses} to restrict the structure of missing faces in $P$ and its face links.
The next result uses the symmetry of stresses on $P$ to show that a missing face in $P$ gives rise to many actual faces.

\begin{lemma} \label{lemma: missing-face-graph-lemma}
Let $P$ be a cs simplicial $d$-polytope with $g_2(P) = {d \choose 2}-d$ and $d \geq 4$.  If $\tau$ is a missing face in $\partial P$ and $3 \leq |\tau| \leq d-1$, then $\left(\tau \setminus e \right) \cup -e$ is a face of $\partial P$ for any edge $e\subset \tau$.  In particular, $G(P)$ contains the graph of the cross-polytope on vertex set $\tau \cup -\tau$.
\end{lemma}

\begin{proof}
Consider the edge $e\subset \tau$ and let $\tau' = \tau \setminus e$.  By Lemma \ref{lemma: missing-face-stress}, there is a stress $\omega$ on $(\st(\tau')\cup\{e\}, \p)$ such that $\omega_e\neq 0$. We can extend $\omega$ to a stress on $(P, \p)$ by assigning zero values to the edges of $P$ that are not in $\st(\tau')\cup\{e\}$. Since all stresses on $P$ are symmetric by Corollary \ref{cor: symmetric-stresses}, we must have $\omega_{-e}=\omega_e\neq 0$, and hence $-e \in \st(\tau')$.  Thus $\tau' \cup -e = (\tau \setminus e) \cup -e$ is a face of $\partial P$, as desired.
\end{proof}

\begin{lemma} \label{lemma: not-d-antipodal-pairs}
Let $P$ be a cs simplicial $d$-polytope with $d \geq 4$.  If there exists a vertex $u$ of $P$ such that $\st(u)$ and $\st(-u)$ share at least $d$ pairs of antipodal vertices, then $g_2(P) > {d \choose 2}-d$.
\end{lemma}

\begin{proof}
We start by establishing some notation. Let $\Sigma \subseteq \partial P$ denote the $(d-1)$-dimensional subcomplex $\st(u) \cup \st(-u)$ and let $\Lambda \subseteq \partial P$ denote the $(d-2)$-dimensional subcomplex $\lk(u) \cup \lk(-u)$.  Let $H$ be the hyperplane through the origin in $\R^d$ whose normal vector is $\p(u)$, and let $\pi: \R^d \rightarrow H$ be the orthogonal projection of $\R^d$ onto $H$.
Perturbing $P$ slightly (without changing its symmetry or combinatorial type) we may also assume that
$\pi$ is injective on the framework $(\st(u), \p)$.

Now we begin the proof of the lemma.  Assume to the contrary that $g_2(P) = {d \choose 2}-d$. By Lemma \ref{lemma: star-of-face-rigid}, the frameworks $(\st(u), \p)$ and $(\st(-u),\p)$ are infinitesimally $d$-rigid. As they share $d$ affinely independent vertices, $(\Sigma,\p)$ is infinitesimally $d$-rigid by the Gluing Lemma.  Since $(\Sigma, \p)$ is also cs and since it is a subframework of $(P, \p)$, we conclude from Corollary \ref{inf-rig-subgraph} that {$g_2(\Sigma, \p) = {d \choose 2}-d$.}

Since $(\st(u), \p)=(u*\lk(u),\p)$ is infinitesimally $d$-rigid, it follows from the Cone Lemma that the framework $(\lk(u), \pi\circ\p)$ is infinitesimally $(d-1)$-rigid in $H$.   Similarly, since $\p(-u)=-\p(u)$ is also a normal vector to $H$, the framework $(\lk(-u), \pi\circ\p)$ is also infinitesimally $(d-1)$-rigid in $H$.  Further, since the vertices of $(\st(u)\cap\st(-u),\p)=(\lk(u)\cap\lk(-u), \p)$ affinely span $\R^d$, their projections span $H$. Hence $(\Lambda, \pi\circ\p)$ is infinitesimally $(d-1)$-rigid in $H$ by the Gluing Lemma. As $(\Lambda, \pi\circ\p)$ is also cs, Lemma \ref{lemma: g2-lower-bound} implies that
\begin{equation} \label{g2-of-Lambda}
 f_1(\Lambda) - (d-1)f_0(\Lambda) + {d \choose 2}=g_2(\Lambda,\pi \circ \p) \geq {d-1 \choose 2}-(d-1).
\end{equation}

Further,
\begin{eqnarray*}
f_1(\Sigma) &=& f_1(\Lambda) + f_0(\Lambda) + f_0(\lk(u) \cap \lk(-u)) \\
&\geq& f_1(\Lambda) + f_0(\Lambda) + 2d,
\end{eqnarray*}
since $\lk(u)$ and $\lk(-u)$ share at least $d$ pairs of antipodal vertices. Therefore,
\begin{eqnarray*}
g_2(\Sigma,\p) &=& f_1(\Sigma) - d \cdot f_0(\Sigma) + {d+1 \choose 2} \\
&\geq& \left(f_1(\Lambda) + f_0(\Lambda) + 2d\right)  - d \cdot (f_0(\Lambda) + 2) + {d+1 \choose 2} \\
&=& g_2(\Lambda,\pi \circ \p) + d \\
&\geq& {d-1 \choose 2} - (d-1) + d = {d \choose 2}-d+2.
\end{eqnarray*}
Here, the fourth line comes from Eq.~\eqref{g2-of-Lambda}. This contradicts our previous calculation showing {$g_2(\Sigma, \p) = {d \choose 2}-d$.}
\end{proof}

\section{More on missing faces in $P$}  \label{section:6}

\subsection{Swartz's operation and missing faces in vertex links}

In addition to our previous reduction to the case that $P$ is prime (see Section 4), in this subsection we will further show that if $P$ is prime with $g_2(P)={d \choose 2}-d$, then $\lk(u)$ is prime for every vertex $u$ in $P$.  This requires the following operation introduced by Swartz in \cite{Swartz-finiteness}.

Let $\Delta$ be a prime simplicial $(d-1)$-sphere and assume there exists a vertex $v_0 \in V(\Delta)$ such that $\lk_{\Delta}(v_0)$ is not prime. Then $\lk_{\Delta}(v_0)$ contains a missing $(d-2)$-face $\tau$.  Note that $\tau \notin \Delta$ --- if $\tau$ were in $\Delta$, then $\tau \cup v_0$ would be a missing facet of $\Delta$, contradicting the assumption that $\Delta$ is prime.  Thus $\lk_{\Delta}(v_0)$ can be decomposed as the connected sum of two simplicial spheres: $\lk_{\Delta}(v_0) = S_1 \#_{\partial\overline{\tau}}S_2$.  Form a new simplicial $(d-1)$-sphere $\Delta'$ as follows. First,  remove $v_0$ from $\Delta$ and introduce two new vertices $x$ and $y$; then add the face $\tau$ to $\Delta$, along with the subcomplexes $x*S_1$ and $y*S_2$.  In other words, replace the ball $\st_\Delta(v_0)=v_0 * (S_1 \#_{\partial \overline{\tau}}S_2)$ with the ball $(x*S_1) \cup (y*S_2)$.

Let us return to our cs simplicial $d$-polytope $P$ with $d \geq 4$.  Assume $P$ is prime but $\lk(v_0)$ has a missing facet $\tau$, and decompose $\lk(v_0)$ as $S_1 \#_{\partial\overline{\tau}}S_2$.  Then $\lk(-v_0)$ also has a missing facet $-\tau$, and so $\lk(-v_0) = (-S_1) \#(-S_2)$ (glued along the boundary of $-\tau$).  Let $\Gamma$ be the simplicial complex obtained from $\partial P$ by applying Swartz's operation first to $v_0$,  then to $-v_0$, and introducing four new vertices $x, y, -x$, and $-y$.  Further, modify $\p: V(P) \rightarrow \R^d$, the map given by the vertex coordinates of $P$, as $\widetilde\p: V(\Gamma) \rightarrow \R^d$ by defining $\widetilde\p(x) = \widetilde\p(y) = \p(v_0)$, $\widetilde\p(-x) = \widetilde\p(-y) = \p(-v_0)$, and otherwise $\widetilde\p(w) = \p(w)$.  Note that $\widetilde\p(-x) = -\widetilde\p(x)$ and $\widetilde\p(-y) = -\widetilde\p(y)$ since $\p(-v_0) = -\p(v_0)$, and hence $(\Gamma,\widetilde\p)$ is a cs framework. Our next objective will be to show that this framework is infinitesimally $d$-rigid.  We shall require the following lemmas.

\begin{lemma} \label{lem: connected-components}
Let $P$, $\Gamma$, and $\tau$ be as above. Then the graph $G(\Gamma) \setminus (V(\tau) \cup V(-\tau))$ has at most two connected components.
\end{lemma}

\begin{proof}
Let $\beta_i(-)$ denote the $i$-th reduced Betti number.  Let $\Gamma'$ be the restriction of $\Gamma$ to the vertices in $\tau \cup -\tau$. Then $\Gamma'$ is a subcomplex of $\partial\C^*_{d-1}$, and so $\beta_{d-2}(\Gamma') \leq 1$.  Since $\Gamma$ is a $(d-1)$-sphere, the Alexander Duality Theorem \cite{Hatcher} implies that $\beta_0(\Gamma \setminus (V(\tau) \cup V(-\tau))) = \beta_{d-2}(\Gamma') \leq 1$, yielding the statement.
\end{proof}

\begin{lemma} \label{lem: cone-frameworks}
Let $P$, $v_0$, $\tau$, $S_1$, and $S_2$ be as above. Then both frameworks $(v_0*S_1,\p)$ and $(v_0*S_2,\p)$ are infinitesimally $d$-rigid.
\end{lemma}

\begin{proof}
Since, by Lemma \ref{lemma: star-of-face-rigid}, $(\st_{P}(v_0),\p)=(v_0 * (S_1 \#_{\partial\overline{\tau}}S_2),\p)$ is infinitesimally $d$-rigid, its stress space has dimension
$$
\dim \Stress(v_0 * (S_1 \#_{\partial\overline{\tau}}S_2),\p) = g_2(v_0 * (S_1 \#_{\partial\overline{\tau}}S_2))  = g_2(v_0*S_1) + g_2(v_0*S_2).
$$

On the other hand, we claim that $$\Stress(v_0*S_1,\p) \oplus \Stress(v_0*S_2,\p) \subseteq \Stress(v_0*(S_1 \# S_2),\p).$$ Indeed, any stress on $(v_0*S_i,\p)$ (for $i=1,2$) can be extended to a stress on $(v_0*(S_1 \# S_2),\p)$ by assigning a weight of 0 to any unused edge.  Further, $\Stress(v_0*S_1,\p) \cap \Stress(v_0*S_2,\p) = \Stress(v_0*(S_1 \cap S_2),\p) = \Stress(v_0*\overline{\tau},\p)$. Since $v_0*\overline{\tau}$ is a $(d-1)$-simplex, our genericity assumption on the vertices of $P$ implies that any stress on $(v_0*\overline{\tau},\p)$ is trivial. Therefore,
\begin{eqnarray*}
\dim \Stress(v_0*(S_1 \# S_2),\p) &\geq& \dim \Stress(v_0*S_1,\p) + \dim \Stress(v_0*S_2,\p) \\
&\geq& g_2(v_0*S_1) + g_2(v_0*S_2) \\
&=& g_2(v_0*(S_1 \# S_2)).
\end{eqnarray*}
Equality must hold throughout the above equation array, and $\dim \Stress(v_0*S_i,\p) = g_2(v_0*S_i)$ if and only if $(v_0*S_i,\p)$ is infinitesimally $d$-rigid.
\end{proof}

\begin{theorem} \label{thm: thmB}
Let $P$, $v_0$, $\tau$, $\Gamma$, $\widetilde\p$, $S_1$, and $S_2$ be as above.  The framework $(\Gamma,\widetilde\p)$ is infinitesimally $d$-rigid.
\end{theorem}

\begin{proof}
Let $w$ be a vertex of $\Gamma$ that does not belong to $\tau \cup -\tau \cup \{x,y,-x,-y\}$. If $w$ is a vertex of $S_1\cup S_2$ or $-S_1\cup -S_2$, then the definition of Swartz's operation implies that $v_0\in \lk_P(w)$ or $-v_0\in \lk_P(w)$ respectively. Since $\widetilde{\p}(x)=\widetilde{\p}(y)=\p(v_0)$ and  $\widetilde{\p}(-x)=\widetilde{\p}(-y)=\p(-v_0)$, the frameworks $(\st_\Gamma(w),\widetilde{\p})$ and $(\st_P(w),\p)$ are isomorphic. If instead $w$ is not a vertex of $S_1\cup S_2$ or $-S_1\cup -S_2$, then $(\st_\Gamma(w),\widetilde\p) = (\st_{P}(w),\p)$. Thus, in all cases, $(\st_\Gamma(w),\widetilde\p)$ is infinitesimally $d$-rigid by Lemma \ref{lemma: star-of-face-rigid}.  Similarly, the stars of $v_0$ and $-v_0$ in $\partial P$ are infinitesimally rigid  under the embedding $\p$, so Lemma \ref{lem: cone-frameworks} implies that the stars of $x, y, -x, $ and $-y$ in $\Gamma$ are infinitesimally rigid under $\widetilde\p$.

Next, let $K$ be a connected component of $\Gamma \setminus (\tau \cup -\tau)$, and order the vertices of $K$ as $v_1, \ldots, v_m$ so that for each $2 \leq k \leq m$, vertex $v_k$ has a neighbor $v_i$ with $i<k$ (ordering the vertices by breadth first search, for example, will accomplish this). This ensures that $\st_\Gamma(v_i)$ and $\st_\Gamma(v_k)$ intersect along a facet $\sigma$ of $\Gamma$ that contains the edge $\{v_i,v_k\}$.  As $\sigma$ is a facet of $\Gamma$, the vectors in $\widetilde\p(\sigma)$ are affinely independent. It follows by repeated applications of the Gluing Lemma that $(G_K,\widetilde\p)$ is infinitesimally $d$-rigid, where $G_K$ denotes the graph of $\bigcup_{w \in K} \st_\Gamma(w)$.

Finally, by Lemma \ref{lem: connected-components} we know $\Gamma \setminus (\tau \cup -\tau)$ has at most two connected components.  If it is connected, then the computation in the previous paragraph shows $(\Gamma,\widetilde{\p}) = (G_K, \widetilde{\p})$ is infinitesimally $d$-rigid.  Otherwise, suppose $\Gamma \setminus (\tau \cup -\tau)$ has two connected components $K$ and $K'$.  As in the proof of Lemma \ref{lem: connected-components}, let $\Gamma'$ denote the restriction of $\Gamma$ to the vertices in $\tau \cup -\tau$.  It follows from that proof that $\Gamma'$ is $\partial \C^*_{d-1}$ and it forms a separating codimension-$1$ sphere  in $\Gamma$.  Since $\widetilde\p(\tau) \cup \widetilde\p(-\tau)$ affinely span a $(d-1)$-dimensional space, $(\Gamma,\widetilde{\p}) = (G_K \cup G_{K'}, \widetilde{\p})$ is infinitesimally $d$-rigid by the Gluing Lemma.
\end{proof}

\begin{corollary} \label{cor: prime-links}
Let $P$ be a prime cs simplicial $d$-polytope with $d \geq 4$ and $g_2(P) = {d \choose 2}-d$, and let $u$ be a vertex of $P$.  Then $\lk(u)$ is prime.
\end{corollary}

\begin{proof}
Assume to the contrary that $\lk(u)$ contains a missing $(d-2)$-face.  Each application of Swartz's operation increases the number of edges by $d-1$ and increases the number of vertices by one, and hence decreases $g_2$ by one.  Therefore, $g_2(\Gamma)  = g_2(P) - 2 < {d \choose 2}-d$.  However, by Theorem \ref{thm: thmB}, the cs framework $(\Gamma,\widetilde\p)$ is infinitesimally $d$-rigid.  This contradicts Lemma \ref{lemma: g2-lower-bound}.
\end{proof}

\subsection{General results on missing faces in $P$}

\begin{lemma} \label{lemma: no-missing-faces}
Let $P$ be a prime cs simplicial $d$-polytope with $d \geq 5$ and $g_2(P) = {d \choose 2}-d$.  There is no face $\tau$ of $\partial P$ with $|\tau| \leq d-2$ whose  link is the boundary of a simplex.
\end{lemma}

\begin{proof}
Assume to the contrary that such a face $\tau$ exists and $\lk(\tau) = \partial \overline\sigma$.  Fix $u \in \sigma$.  We will show $u$ is adjacent to every vertex in the sets $\tau, -\tau, \sigma \setminus u$, and $-(\sigma \setminus u)$.  From this, it follows that $\st(u)$ and $\st(-u)$ share the vertices in $\tau\cup \sigma \setminus u$ and their antipodes.  But $|\tau\cup \sigma \setminus u| = d$, which contradicts Lemma \ref{lemma: not-d-antipodal-pairs}.

First note that $u$ is adjacent to every vertex in $\tau$ and every vertex in $\sigma \setminus u$ because $\lk(\tau) = \partial \overline{\sigma}$.

Observe that $\overline{\tau} *\partial\overline{\sigma} \subseteq \partial P$, but the set $\tau \cup \sigma$ has $d+1$ vertices and hence cannot be a face of $\partial P$ (its dimension is too large), nor can it be a missing face of $\partial P$ (otherwise, $P$ itself would be the $d$-simplex, which is not centrally symmetric).  Thus there exists a proper face $\tau_1 \subsetneq \tau$ such that $\tau_1 \cup \sigma$ is a missing face in $\partial P$. Note that $|\tau_1 \cup \sigma| \leq d-1$ since $ P$ is prime and $|\tau_1 \cup \sigma|  \geq |\sigma| = d+1-|\tau| \geq 3$. Hence we can apply Lemma \ref{lemma: missing-face-graph-lemma} to the missing face $\tau_1 \cup \sigma$, which implies that $u$ is adjacent to every vertex in $-\tau_1$ and every vertex in $-(\sigma \setminus u)$.

It remains to show $u$ is adjacent to every vertex in $-(\tau \setminus \tau_1)$. Fix a vertex $v \in \tau \setminus \tau_1$. Since $\partial \overline{\sigma} \subseteq \lk(v)$ and since $\tau_1 \cup \sigma$ is a missing face in $\partial P$, there exists $\tau_2\subseteq \tau_1$ (possibly empty) such that $\sigma':=\tau_2 \cup \sigma$ is a missing face in $\lk(v)$.

Let $e$ be any edge in $\sigma$ containing $u$.  Then $\sigma'$ is a missing face in $\lk(v)$, and hence also in $\st(v)$, containing $e$.  Since $\lk(v)$ is prime (Corollary \ref{cor: prime-links}), $|\sigma'| \leq d-2$.
Therefore, $\sigma'\cup v \setminus e$ is a face of $\partial P$ with $|\sigma' \cup v \setminus e| \leq d-3$. Thus,
$(\st_{\st(v)}(\sigma'\setminus e), \p)=(\st_{\partial P}(\sigma'\cup v \setminus e), \p)$ is infinitesimally rigid by Lemma \ref{lemma: star-of-face-rigid}. As $\sigma'$ is a missing face of $\st(v)$ and $|\sigma'|\geq |\sigma|\geq 3$, it follows from Lemma \ref{lemma: missing-face-stress}(2) that there is a stress on $(\st_{\st(v)}(\sigma'\setminus e)\cup\{e\},\p)$ (and hence on $(P, \p)$) that is non-zero on $e$.  Since all stresses on $(P, \p)$ are symmetric by Corollary \ref{cor: symmetric-stresses}, we conclude that this stress is also non-zero on $-e$. Thus, $-e$ must be  an edge of $\st(v)$, and so $\{-u,v\}$ is an edge of $P$.  By central symmetry, $\{u,-v\}$ is also an edge, that is, $u$ is adjacent to $-v$.  This completes the proof that $u$ is adjacent to every vertex in $-(\tau \setminus \tau_1)$.
\end{proof}

The following corollary is immediate.

\begin{corollary} \label{cor: links-not-stacked}
Let $P$ be a prime cs $d$-polytope with $d \geq 5$ and $g_2(P) = {d \choose 2}-d$.  If $\tau$ is a face of $\partial P$ with $|\tau| \leq d-3$, then $\lk(\tau)$ is not stacked.
\end{corollary}

\begin{proof}
Assume to the contrary that $\lk(\tau)$ is stacked for some face $\tau\in\partial P$ with $|\tau| \leq d-3$.  Then there exists a vertex $u \in \lk(\tau)$ such that $\lk_{\lk(\tau)}(u) = \lk_{P}(\tau \cup u)$ is the boundary of a simplex.  This contradicts Lemma \ref{lemma: no-missing-faces}.
\end{proof}

\section{Completing the proof of Theorem \ref{main-pre-thm}}  \label{section:7}

In this section we continue to restrict our attention to prime cs $d$-polytopes with $d\geq 5$ and $g_2={d \choose 2}-d$. The next proposition implies the following counterpart of Lemma \ref{lemma: not-d-antipodal-pairs}: the stars of any two antipodal vertices in such a polytope share at least $2(d-1)$ common vertices.

\begin{proposition} \label{prop: antipodal-pairs-in-links}
Let $P$ be a prime cs $d$-polytope with $d \geq 5$ and $g_2(P) = {d \choose 2}-d$.  If $k \geq 4$ and $\tau \in \partial P$ is a face of size $d-k$, then $\lk(\tau)$ contains at least $k$ pairs of antipodal vertices.
\end{proposition}

\begin{proof} We prove the claim by induction on $k$. When $k=4$, $\lk(\tau)$ is the boundary complex of a simplicial $4$-polytope, which, by Corollary \ref{cor: links-not-stacked}, is not stacked. Hence $g_2(\st(\tau))=g_2(\lk(\tau)) > 0$.  Since $(\st(\tau),\p)$ is infinitesimally rigid by Lemma \ref{lemma: star-of-face-rigid}, we infer from Lemma \ref{basic-rig-prop} that $(\st(\tau), \p)$ supports a nontrivial stress. By Corollary \ref{cor: symmetric-stresses}, this stress is symmetric, and hence attains non-zero values only on the edges of $\Lambda:=\lk(\tau) \cap \lk(-\tau)$.  If $\Lambda$ had only 3 pairs of antipodal vertices, say $v_i,-v_i$ for $i=1,2,3$, the framework $(\Lambda, \p)$ would be a subgraph of the graph of the $3$-dimensional cross-polytope $\conv\{\pm\p(v_1),\pm\p(v_2),\pm\p(v_3)\}$, and so it would not support any nontrivial stresses. A similar argument would apply if there were fewer than 3 pairs. Therefore, $\lk(\tau) \cap \lk(-\tau)$ contains 4 or more pairs of antipodal vertices, which establishes the base case.

Now suppose $k > 4$. Let $v$ be a vertex of $\lk(\tau)$.  Since $\lk(\tau) \supset \lk(\tau \cup v)$ and $|\tau \cup v| = d-(k-1)$, the inductive hypothesis implies that $\lk(\tau \cup v)$ contains at least $k-1$ pairs of antipodal vertices.  Let $\{u,-u\}$ be one such pair.  Applying the inductive hypothesis again to $\lk(\tau) \supset \lk(\tau \cup u)$ shows that $\lk(\tau \cup u)$ contains at least $k-1$ pairs of antipodal vertices.  This, together with $\{u,-u\}$ exhibits at least $k$ pairs of antipodal vertices in $\lk(\tau)$ and completes the proof.
\end{proof}

Proposition \ref{prop: antipodal-pairs-in-links} and Lemma \ref{lemma: not-d-antipodal-pairs} imply that $u$ and $-u$ have exactly $2d-2$ common neighbors for every vertex $u$ of $P$.  This proves the first part of Theorem \ref{main-pre-thm}.

\begin{corollary} \label{cor: star-union-rigid}
Let $P$ be a prime cs $d$-polytope with $d \geq 5$ and $g_2(P) = {d \choose 2}-d$. Then for every vertex $u$ of $P$, $(\st(u)\cup \st(-u), \p)$ is infinitesimally rigid.
\end{corollary}

\begin{proof}
By Proposition \ref{prop: antipodal-pairs-in-links}, $(\lk(u), \p)$ contains at least $d-1$ pairs of antipodal vertices.  Therefore, $(\st(u)\cap\st(-u), \p)$ affinely spans a subspace of dimension at least $d-1$.  The vertex stars $(\st(u),\p)$ and  $(\st(-u),\p)$ are infinitesimally rigid by Lemma \ref{lemma: star-of-face-rigid}, and so $(\st(u) \cup \st(-u), \p)$ is infinitesimally rigid by the Gluing Lemma.
\end{proof}

\begin{proposition} \label{prop: cs-rigid-is-whole-graph}
Let $P$ be a prime cs $d$-polytope with $d \geq 5$ and $g_2(P) = {d \choose 2}-d$.  Let $G'$ be a subgraph of the graph of $P$ such that the framework $(G', \p)$ is cs infinitesimally rigid and affinely spans $\R^d$. Then $G'$ is the graph of $P$.
\end{proposition}

\begin{proof}
Our assumptions on $G'$ along with Corollary \ref{inf-rig-subgraph} imply that $g_2(G') = {d \choose 2}-d$ and $\Stress(G', \p) = \Stress(P,\p)$.

We claim that $V(G')=V(P)$.  Suppose to the contrary that there exists a vertex $u$ of $P$ that does not belong to $V(G')$.  By Proposition \ref{prop: antipodal-pairs-in-links}, there is a pair of antipodal vertices $v, -v$ in $\lk(u)$.  This implies that the edges $\{u,v\}$ and $\{u,-v\}$ belong to $\partial P$ and hence by central symmetry, the $4$-cycle on vertices $u, v, -u, -v$ is a subgraph of $P$. Moreover, this cycle is induced since $\{u,-u\}$ and $\{v,-v\}$ are antipodal pairs and hence non-edges.  By Lemma \ref{lemma: missing-face-stress}, there is a stress on $(P, \p)$ that is nonzero on the edge $e = \{u,v\}$. However, as $e$ is an edge of $P$ but not $G'$, this means there exists a stress on $(P, \p)$ that does not belong to $\Stress(G', \p)$.  This is a contradiction.

Therefore, $G'$ and $P$ have the same number of vertices.  Since $g_2(G') = g_2(P)$, this also implies that $G'$ and $P$ have the same number of edges.  Hence $G' = G(P)$.
\end{proof}

Together, Corollary \ref{cor: star-union-rigid} and Proposition \ref{prop: cs-rigid-is-whole-graph} complete the proof of Theorem \ref{main-pre-thm}, and hence also of our main result, Theorem \ref{main-thm'}.


\section{Concluding remarks and open problems}
\subsection{Towards the Lower bound Theorem for cs simplicial spheres}
Many problems related to the Lower Bound Theorem for cs simplicial complexes remain wide open. For instance, we strongly suspect that Stanley's inequality on the $g_2$-number of cs polytopes and our characterization of the minimizers continue to hold in the generality of cs homology spheres or perhaps even cs normal pseudomanifolds:

\begin{conjecture} \label{g2_spheres}
Let $\Delta$ be a cs simplicial complex of dimension $d-1\geq 3$. Assume further that $\Delta$ is a homology sphere (or a connected homology manifold or even a normal pseudomanifold). Then $g_2(\Delta)\geq {d \choose 2}-d$. Furthermore, equality holds if and only if $\Delta$ is the boundary complex of a cs $d$-polytope obtained from the cross-polytope $\C^*_d$ by symmetric stacking.
\end{conjecture}

\begin{remark}
	The cs triangulation of $\mathbb{S}^i\times \mathbb{S}^{d-i-1}$ constructed in \cite{Klee-Novik} has $g_2=\binom{d+1}{2}$ for all $1\leq i\leq \frac{d-1}{2}$. However, there might be a better lower bound on $g_2$ for  connected cs homology manifolds that involves the first Betti number.
\end{remark}

Similarly to the classical non-cs case, by Lemma \ref{lemma: g2-lower-bound}, the following conjecture would imply the inequality part of Conjecture \ref{g2_spheres}.

\begin{conjecture} \label{inf_rig_cs_spheres}
Let $\Delta$ be a cs simplicial complex of dimension $d-1\geq 3$. Assume further that $\Delta$ is a homology sphere (a connected homology manifold or even a normal pseudomanifold). Then there exists a map $\p: V(\Delta)\to \R^d$ such that $(\Delta,\p)$ is a cs framework that is infinitesimally rigid in $\R^d$.
\end{conjecture}

One of the reasons Conjecture \ref{inf_rig_cs_spheres} appears to be hard is that the links of cs complexes are usually not centrally symmetric, and so the standard inductive arguments with the Cone and Gluing Lemmas do not apply. However, if $\Delta$ is a cs complex, $\tau$ a face of $\Delta$, and $\lk_\Delta(\tau)$ contains two antipodal vertices of $\Delta$, then these two vertices do not form an edge in the link. This leads to the following conjecture on $2$-dimensional simplicial complexes that, if true, would imply Conjecture \ref{inf_rig_cs_spheres}.

\begin{conjecture} \label{partial-match}
Let $\Gamma$ be a 2-dimensional simplicial sphere (or even a simplicial manifold), and let $\{\{u_1,w_1\}, \ldots, \{u_m,w_m\}\}$ be a collection of pairwise disjoint missing edges of $\Gamma$ (possibly empty). Then there exists a map $\p:V(\Gamma)\to \R^3$ such that $\p(u_i)=-\p(w_i)$ for all $i=1,\ldots,m$ and the framework $(\Gamma, \p)$ is infinitesimally rigid.
\end{conjecture}

{An indication that Conjecture \ref{partial-match} might be difficult comes from the following observation: while every simplicial $2$-sphere can be realized as the boundary complex of a $3$-polytope, there exists a simplicial $2$-sphere $\Gamma$ and a collection $\{\{u_1,w_1\},...,\{u_m,w_m\}\}$ as in the statement of the conjecture such that for every realization $\p$ of $\Gamma$ as the boundary complex of a polytope, $\p(u_j) \neq -\p(w_j)$ for some $1\leq j\leq m$. As an example, consider the connected sum of the boundary complex of the cross-polytope on the vertex set $\{u_1,u_2,u_3,w_1,w_2,w_3\}$ with the boundary complexes of the simplices on the vertex sets $\{u_1,u_2,u_3,u_4\}$ and $\{w_1,u_2,u_3, w_4\}$, respectively.}

\subsection{Higher $g$-numbers}
The definition of $g_2=f_1-df_0+{d+1 \choose 2}$ naturally extends to the notion of higher $g$-numbers: for a $(d-1)$-dimensional simplicial complex $\Delta$, the $h$-polynomial, $h(\Delta,x)=\sum_{i=0}^d h_i(\Delta)x^{d-i}$, is defined by
\[\sum_{i=0}^d h_i(\Delta)x^{d-i}=\sum_{j=0}^d f_{j-1}(\Delta) (x-1)^{d-j}.\]
Here $f_{j-1}(\Delta)$ denotes the number of $(j-1)$-dimensional faces in $\Delta$; in particular, {$f_{-1}(\Delta)=1$.} The $g$-numbers of $\Delta$ are then defined as $g_r(\Delta):=h_r(\Delta)-h_{r-1}(\Delta)$; that is, $$g_r(\Delta)=\sum_{j=0}^r (-1)^{r-j}f_{j-1}(\Delta){d-j+1 \choose r-j}.$$ For a simplicial $d$-polytope $P$, we write $h_r(P)$ and $g_r(P)$ instead of $h_r(\partial P)$ and $g_r(\partial P)$, respectively.  Since $h(\C^*_d,x) = (1+x)^d$, it follows that $g_r(\C^*_d) = {d \choose r} - {d \choose r-1}$ for all $r \leq \lfloor \frac{d}{2} \rfloor$.

It follows from the $g$-theorem \cite{Stanley-gthm} that if $P$ is a simplicial $d$-polytope with $g_r(P) = 0$ for some $1 \leq r < \halfd$, then $g_{r+1}(P) = 0$.  This, together with Stanley's \cite{Stanley-cs} result that  $g_r(P)\geq {d \choose r}-{d \choose r-1}$ for a cs simplicial $d$-polytope, motivates the following conjecture. (Note that the case of $r=1$ is easy and the main result of this paper establishes the case of $r=2$.)

\begin{conjecture} \label{g_r}
Let $P$ be a  cs simplicial $d$-polytope. If $g_r(P) = g_r(\C^*_d)$ for some $3 \leq r < \halfd$, then $g_{r+1}(P) = g_{r+1}(\C^*_d)$.
\end{conjecture}

In view of the equality case of the Generalized Lower Bound Theorem due to Murai and Nevo \cite{Murai-Nevo}, it is natural to posit the following generalization of Theorem \ref{main-thm}, which would imply Conjecture \ref{g_r}. We refer our readers to Ziegler's book \cite[Section 8.1]{Ziegler} for the definition of a polytopal complex. We also recall that the $i$-skeleton of a simplicial complex $\Delta$ is $\skel_i(\Delta):= \{ \tau\in\Delta \, : \, \dim\tau \leq i\}$.

\begin{conjecture} \label{r-stackedness}
Let $P$ be a  cs simplicial $d$-polytope, and assume that $g_r(P) = {d \choose r}-{d \choose r-1}$ for some $3 \leq r \leq \halfd$. Then there exists a unique polytopal complex $\C$ in $\R^d$ with the following properties: (i) one of the faces of $\C$ is the cross-polytope $\C^*_d$, all other faces of $\C$ are simplices that come in antipodal pairs; (ii) $\C$ is a ``cellulation" of $P$, that is, $\bigcup_{C\in\C}C=P$, and (iii) each element $C\in \C$ of dimension $\leq d-r$ is a face of $P$. Furthermore, the collection of simplices of $\C$ consists of all proper faces of $P$ along with all simplices $\conv(U)$ with $U\subset V(P)$, such that the $(d-r)$-skeleton of $\overline{U}$ is contained in $\partial P$.
\end{conjecture}

Assuming the existence part of Conjecture \ref{r-stackedness}, the proof of the uniqueness and of the furthermore-part of this conjecture is very similar to the proofs of the analogous statements in the non-cs case, see {\cite[Thm.~2.20]{BagchiDatta-stellated}, \cite[Thm.~2.3]{Murai-Nevo}, and \cite[Thm.~5.17]{KleeNovik-balancedLBT}.} Indeed, let $\C$ be a complex satisfying conditions (i)--(iii) of the conjecture. As in the proof of {\cite[Thm.~5.17]{KleeNovik-balancedLBT}}, introduce a new vertex $w$ and replace the (unique) cross-polytopal face of $\C$ with a cone with apex $w$ over the boundary complex of this face. The resulting complex $B$ is a simplicial $d$-ball. Introduce one additional new vertex $v_0$ and let $\Lambda=B\cup (v_0 * \partial B)$ be the corresponding simplicial $d$-sphere. The proof of the furthermore-part now follows using the standard tools such as Alexander duality and the Mayer--Vietoris sequence. We omit the details.

We close {this subsection} with a proof of the existence part of Conjecture \ref{r-stackedness} in a certain special case.  In \cite{Stanley-sd}, Stanley studies the effect of subdivisions of simplicial complexes on their face numbers. In particular, for a $(d-1)$-dimensional complex $\Gamma$ that provides a subdivision of a $(d-1)$-simplex $\overline{V}$, Stanley introduces the notion of the {\em local $h$-vector} $\ell_V(\Gamma)=(\ell_0,\ldots,\ell_d)$ or  {\em local $h$-polynomial}
$\ell_V(\Gamma, x) =\ell_0+\ell_1 x + \cdots + \ell_d x^d$. He then proves that if $\Gamma$ is combinatorially equivalent to a {\em regular} subdivision of the simplex, then the vector $\ell_V(\Gamma)$ is non-negative, symmetric (that is, $\ell_i=\ell_{d-i}$ for all $i$)  and unimodal.

Given any subdivision $\Delta'$ of $\Delta$ (here both $\Delta$ and $\Delta'$ are simplicial complexes) and an arbitrary face $\tau$ of $\Delta$, one obtains an induced subdivision $\Delta'_\tau$ of $\overline{\tau}$ --- the restriction of $\Delta'$ to $\overline{\tau}$. Stanley \cite{Stanley-sd} furthermore proves that
\begin{equation} \label{ell-h}
 h(\Delta', x)=\sum_{\tau\in\Delta}\ell_{\tau}(\Delta'_\tau,x) h(\lk_\Delta(\tau),x).
\end{equation}

We say that a subdivision $\Delta'$ of $\Delta$ is {\em Stanley-regular} if for every face $\tau$ of $\Delta$, $\Delta'_\tau$ is combinatorially equivalent to a regular subdivision of the simplex $\overline{\tau}$. In particular, for a vertex $v$ of $\Delta$, $\Delta'_v$ is a single vertex of $\Delta'$; we identify this vertex with $v$ and refer to all vertices of $\Delta'$ that are not vertices of $\Delta$ as {\em new vertices}. We note that for each new vertex $v'$, there is a unique face $\tau$ of $\Delta$ such that $v'$ is an interior vertex of the ball $\Delta'_\tau$; this face $\tau$ is called the {\em carrier} of $v'$.

Using Stanley's results on the local $h$-vectors of regular subdivisions we are now in a position to prove the following special case of Conjecture \ref{r-stackedness}.

\begin{proposition}
Let $P$ be a cs simplicial $d$-polytope such that $\Delta':=\partial P$ is a Stanley-regular subdivision of $\Delta:=\partial C^*_d$ that respects central symmetry. If $g_r(P)=g_r(\C^*_d)$ for some $3\leq r\leq d/2$, then $P$ has a cellulation as in Conjecture \ref{r-stackedness}.
\end{proposition}

\begin{proof}
Let $\tau$ be a face of $\Delta$. Then the link of $\tau$ in $\Delta$ is the boundary complex of a $(d-|\tau|)$-dimensional cross-polytope, and so $h(\lk_\Delta(\tau),x)=(1+x)^{d-|\tau|}$. We now concentrate on $\ell_{\tau}(\Delta'_\tau, x)=\sum_{i=0}^{|\tau|}\ell_i x^i$.  We claim that if $1\leq \dim\tau \leq d-r$, then $\tau$ is not a carrier of any new vertex. Indeed, if $\tau$ is a carrier of a new vertex, then $\Delta'_\tau$ has at least one interior vertex, and so by \cite[Example 2.3(f)]{Stanley-sd}, $\ell_1\geq 1$ while $\ell_0=0$. Since $\ell_{\tau}(\Delta'_\tau, x)$ is non-negative, symmetric, and unimodal, and since $\ell_1>\ell_0$, it is not hard to see from our assumptions on $\dim\tau$ and $r$ (say,  by using \cite[eq.~(1)]{Andrews}) that the coefficient of $x^r$ in  $\ell_{\tau}(\Delta'_\tau,x) h(\lk_\Delta(\tau),x)=\ell_{\tau}(\Delta'_\tau,x) (1+x)^{d-|\tau|}$ is strictly larger than that of $x^{r-1}$. Since for all other faces $\sigma\in\Delta$, $\ell_{\sigma}(\Delta'_\sigma, x)$ is also non-negative, symmetric, and unimodal and since  $\ell_{\emptyset}(\Delta'_\emptyset,x) h(\lk_\Delta(\emptyset),x)=h(\C^*_d, x)$, eq.~(\ref{ell-h}) then implies that $g_r(P)>g_r(\C^*_d)$. This however contradicts
the assumption of the proposition. Thus, no face of $\Delta$ of dimension $\leq d-r$ can carry a new vertex, that is, via our identification of vertices, for every face $\tau\in\Delta$ with $\dim\tau\leq d-r$, $\Delta'_\tau$ is the simplex $\overline{\tau}$.

We now decompose $P$ as follows: for each facet $\tau$ of $\Delta$ such that $\Delta'_\tau$ contains at least one new vertex (not necessarily in the interior), let $U$ be the set of vertices of $P$ that are identified with $V(\tau)$, and consider the corresponding geometric $(d-1)$-simplex $\conv(U) \subset P$. (Note that if $\Delta'_\tau$ does not contain a new vertex, then $\conv(U)$ is a facet of $P$.) By the first paragraph of this proof, all faces of dimension at most $d-r$ in these geometric simplices are faces of $P$. Furthermore, these simplices partition $P$ into several simplicial polytopes $P_0, P_1,\ldots, P_{2m}$ one of which, say, $P_0$ is the cross-polytope, and the others come in antipodal pairs: $P_{2i-1}=-P_{2i}$. Since $P$ is the connected sum of these polytopes, since $g_r(P)=g_r(\C^*_d)=g_r(P_0)$, and since $g_r(P_i)\geq 0$ (for all $i$) by the $g$-theorem \cite{Stanley-gthm}, it follows that all polytopes in this decomposition, but $P_0$, satisfy $g_r(P_i)=0$. Thus, by the main result of \cite{Murai-Nevo}, each $P_i$ (for $i>0$) has a triangulation $T_i$ such that $\skel_{d-r}(T_i)=\skel_{d-r}(P_i)\subseteq \skel_{d-r}(P)$. Taking $\C$ consist of $P_0$ along with all the faces of $T_i$ for $i=1,2,\ldots, 2m$ then provides a cellulation of $P$ that satisfies conditions (i)--(iii) of Conjecture \ref{r-stackedness}.
\end{proof}

\subsection{Polytopes with other symmetries}
In this paper we discussed centrally symmetric simplicial polytopes. Our starting point was Stanley's result \cite{Stanley-cs} asserting that a cs simplicial $d$-polytope $P$ satisfies $g_r(P)\geq {d \choose r}-{d \choose r-1}$ for all $1\leq r\leq \lfloor d/2\rfloor$. However, it is worth mentioning that Adin \cite{Adin95} and later Jorge \cite{Jorge2003} showed that Stanley's result can be suitably extended to polytopes with more intricate symmetries. It would be very interesting to check if our techniques can be adapted to provide a characterization of polytopes that minimize $g_2$ in these more general settings.

\section*{Acknowledgements} We are grateful to the referees for carefully reading our paper and providing many helpful suggestions.



{\small
\bibliography{csbib}

\begin{thebibliography}{10}

\bibitem{Adin95}
R.~M. Adin.
\newblock On face numbers of rational simplicial polytopes with symmetry.
\newblock {\em Adv. Math.}, 115(2):269--285, 1995.

\bibitem{AltPer}
A.~Altshuler and M.~A. Perles.
\newblock Quotient polytopes of cyclic polytopes. {I}. {S}tructure and
  characterization.
\newblock {\em Israel J. Math.}, 36(2):97--125, 1980.

\bibitem{Andrews}
G.~E. Andrews.
\newblock A theorem on reciprocal polynomials with applications to permutations
  and compositions.
\newblock {\em Amer. Math. Monthly}, 82(8):830--833, 1975.

\bibitem{AsimowRothI}
L.~Asimow and B.~Roth.
\newblock The rigidity of graphs.
\newblock {\em Trans. Amer. Math. Soc.}, 245:279--289, 1978.

\bibitem{AsimowRothII}
L.~Asimow and B.~Roth.
\newblock The rigidity of graphs. {II}.
\newblock {\em J. Math. Anal. Appl.}, 68(1):171--190, 1979.

\bibitem{BagchiDatta-stellated}
B.~Bagchi and B.~Datta.
\newblock On {$k$}-stellated and {$k$}-stacked spheres.
\newblock {\em Discrete Math.}, 313(20):2318--2329, 2013.

\bibitem{Barnette-LBT-pseudomanifolds}
D.~Barnette.
\newblock Graph theorems for manifolds.
\newblock {\em Israel J. Math.}, 16:62--72, 1973.

\bibitem{Barnette-LBT}
D.~Barnette.
\newblock A proof of the lower bound conjecture for convex polytopes.
\newblock {\em Pacific J. Math.}, 46:349--354, 1973.

\bibitem{Billera-Lee}
L.~Billera and C.~W. Lee.
\newblock A proof of the sufficiency of {M}c{M}ullen's conditions for
  {$f$}-vectors of simplicial convex polytopes.
\newblock {\em J. Combin. Theory Ser. A}, 31(3):237--255, 1981.

\bibitem{Fogelsanger}
A.~Fogelsanger.
\newblock {\em The generic rigidity of minimal cycles}.
\newblock ProQuest LLC, Ann Arbor, MI, 1988.
\newblock Thesis (Ph.D.)--Cornell University; available electronically at \\
  http://www.armadillodanceproject.com/af/cornell/rigidityintro.pdf.

\bibitem{Hatcher}
A.~Hatcher.
\newblock {\em Algebraic Topology}.
\newblock Cambridge University Press, 2002.

\bibitem{Jorge2003}
H.~A. Jorge.
\newblock Combinatorics of polytopes with a group of linear symmetries of prime
  power order.
\newblock {\em Discrete Comput. Geom.}, 30(4):529--542, 2003.

\bibitem{Kalai-rigidity}
G.~Kalai.
\newblock Rigidity and the lower bound theorem. {I}.
\newblock {\em Invent. Math.}, 88(1):125--151, 1987.

\bibitem{Klee-Novik}
S.~Klee and I.~Novik.
\newblock Centrally symmetric manifolds with few vertices.
\newblock {\em Adv. Math.}, 229(1):487--500, 2012.

\bibitem{KleeNovik-balancedLBT}
S.~Klee and I.~Novik.
\newblock Lower bound theorems and a generalized lower bound conjecture for
  balanced simplicial complexes.
\newblock {\em Mathematika}, 62(2):441--477, 2016.

\bibitem{Lee-94}
C.~W. Lee.
\newblock Generalized stress and motions.
\newblock In {\em Polytopes: abstract, convex and computational ({S}carborough,
  {ON}, 1993)}, volume 440 of {\em NATO Adv. Sci. Inst. Ser. C Math. Phys.
  Sci.}, pages 249--271. Kluwer Acad. Publ., Dordrecht, 1994.

\bibitem{Lee-notes}
C.W. Lee.
\newblock The $g$-theorem.
\newblock http://www.ms.uky.edu/$\sim$lee/ma715sp02/notes.pdf, 2002.

\bibitem{Murai-Nevo}
S.~Murai and E.~Nevo.
\newblock On the generalized lower bound conjecture for polytopes and spheres.
\newblock {\em Acta. {M}ath.}, 210:185--202, 2013.

\bibitem{Sanyal-et-al}
R.~Sanyal, A.~Werner, and G.~M. Ziegler.
\newblock On {K}alai's conjectures concerning centrally symmetric polytopes.
\newblock {\em Discrete Comput. Geom.}, 41(2):183--198, 2009.

\bibitem{Stanley-gthm}
R.~P. Stanley.
\newblock The number of faces of a simplicial convex polytope.
\newblock {\em Adv. in Math.}, 35(3):236--238, 1980.

\bibitem{Stanley-cs}
R.~P. Stanley.
\newblock On the number of faces of centrally-symmetric simplicial polytopes.
\newblock {\em Graphs and Combinatorics}, 3:55--66, 1987.

\bibitem{Stanley-sd}
R.~P. Stanley.
\newblock Subdivisions and local {$h$}-vectors.
\newblock {\em J. Amer. Math. Soc.}, 5(4):805--851, 1992.

\bibitem{Swartz-finiteness}
E.~Swartz.
\newblock Topological finiteness for edge-vertex enumeration.
\newblock {\em Adv. Math.}, 219(5):1722--1728, 2008.

\bibitem{Tay}
T.-S. Tay.
\newblock Lower-bound theorems for pseudomanifolds.
\newblock {\em Discrete Comput. Geom.}, 13(2):203--216, 1995.

\bibitem{Tay-et-al}
T.-S. Tay, N.~White, and W.~Whiteley.
\newblock Skeletal rigidity of simplicial complexes. {II}.
\newblock {\em European J. Combin.}, 16:503--523, 1995.

\bibitem{Walkup}
D.~Walkup.
\newblock The lower bound conjecture for {$3$}- and {$4$}-manifolds.
\newblock {\em Acta Math.}, 125:75--107, 1970.

\bibitem{Whiteley-84}
W.~Whiteley.
\newblock Infinitesimally rigid polyhedra. {I}. {S}tatics of frameworks.
\newblock {\em Trans. Amer. Math. Soc.}, 285(2):431--465, 1984.

\bibitem{Zheng-g2}
H.~Zheng.
\newblock A characterization of homology manifolds with {$g_2\le2$}.
\newblock {\em J. Combin. Theory Ser. A}, 153:31--45, 2018.

\bibitem{Ziegler}
G.~M. Ziegler.
\newblock {\em Lectures on polytopes}, volume 152 of {\em Graduate Texts in
  Mathematics}.
\newblock Springer-Verlag, New York, 1995.

\end{thebibliography}
\bibliographystyle{plain}
}
\end{document}